\documentclass[USenglish]{article}	

\usepackage[utf8]{inputenc}				%(only for the pdftex engine)
\usepackage[big,online]{dgruyter}	%values: small,big | online,print,work
\usepackage{lmodern} 
\usepackage{microtype}
\usepackage[numbers,square,sort&compress]{natbib}

\newtheorem{theorem}{Theorem}
\newtheorem{lemma}{Lemma}

\usepackage{amsmath}
\numberwithin{equation}{section}

\usepackage{bm}
\usepackage{enumerate}
\graphicspath{{figs}}

\begin{document}

%%%--------------------------------------------%%%
\articletype{Research Article}
\received{June 27, 2026}
\journalname{arXiv}
\journalyear{2026}
\startpage{1}
%%%--------------------------------------------%%%

\title{Splitting schemes for problems with memory}
\runningtitle{Splitting schemes for problems with memory}
%\subtitle{Insert subtitle if needed}

\author*[1]{P.N.~Vabishchevich}
\runningauthor{P.N.~Vabishchevich}
\affil[1]{\protect\raggedright 
Lomonosov Moscow State University, 1, Leninskie Gory,  Moscow, Russia; 
North-Caucasus Federal University, 1, Pushkin Street, Stavropol,  Russia, e-mail: vab@cs.msu.ru}
	
%\communicated{...}
\dedication{the memory of a remarkable scientist and friend, Professor Peter Minev.}
	
\abstract{
The paper considers the Cauchy problem for a first-order integro-differential equation with memory in a finite-dimensional Hilbert space. The main computational difficulty of such problems is the need to store and process the solution at all previous time levels. To overcome this difficulty, an approach is used that approximates the memory kernel by a sum of exponentials, which reduces the original nonlocal problem to a local one --- a system of weakly coupled evolution equations with additional ordinary differential equations for auxiliary functions. The problem is formulated in vector form on the direct sum of Hilbert spaces. Unconditional stability of two-level operator-difference schemes with weights is proved under standard restrictions. Splitting schemes are proposed and investigated by separating the local and integral operators of the problem. Possibilities for constructing similar schemes for other nonlocal problems, in particular for the equation with memory of the time derivative of the solution, are noted.
}


\maketitle
\section{Introduction}\label{sec:1}

The classical approach to modeling non-stationary processes is based on local mathematical models. The basic ones are parabolic equations describing dissipation and hyperbolic equations intended for wave processes \cite{LionsBook,evans2010partial}. However, they do not take into account effects that depend on the history of the system.
Modern nonlocal dynamic models are built on the basis of integro-differential equations linking past and current states of the system \cite{GripenbergBook1990,pruss2013evolutionary}. Such models allow describing both dissipative and oscillatory processes, combining features of parabolic and hyperbolic systems.

When solving boundary value problems for equations with memory, traditional finite-element or finite-volume approximations in space are used \cite{KnabnerAngermann2003,QuarteroniValli}. The result is a Cauchy problem for evolutionary operator equations with memory in finite-dimensional Hilbert spaces. The choice of time approximation, which determines accuracy and computational costs, is of key importance. For the numerical solution of first-order evolution equations with memory, quadrature formulas for the integral term \cite{ChenBook1998} and two-level difference schemes for the time derivative \cite{mclean1993numerical,mclean1996discretization} are typically used. A distinctive feature of such algorithms is the need to store and process the solution at all previous time instants, which leads to a significant increase in computational costs.

An effective way to overcome computational difficulties is to replace nonlocal models with their local analogues, which dramatically reduces memory and computation requirements. Such a transition to simpler problems is achieved by selecting special approximations of the difference kernel appearing in the integral terms \cite{linz1985analytical}. In many cases, the most convenient approximation is the sum of exponential functions. As a result, the original integro-differential equations are reduced to a system of local weakly coupled evolution equations, which greatly simplifies calculations.
Issues of justification of local computational algorithms for the approximate solution of the Cauchy problem for integro-differential equations with memory (including approximation of kernels, properties of solutions of the corresponding local differential problems) have been considered in many works.
In \cite{vabMemory}, the approximate solution of the Cauchy problem for a first-order integro-differential equation with memory in the solution (in a real finite-dimensional Hilbert space) was studied when approximating the kernel by a sum of exponentials. Problems with memory of the time derivative of the solution in Hilbert spaces are discussed in \cite{vabishchevich2023approximate}. In \cite{vabishchevich2025operator}, stable operator-difference schemes for a second-order Volterra integro-differential equation are constructed.

The computational implementation of unconditionally stable implicit schemes for the approximate solution of non-stationary problems can be facilitated by using splitting schemes \cite{Marchuk1990,VabishchevichAdditive}. In this case, the transition to a new time level is ensured by solving subproblems based on an additive representation of the problem operator. In particular, when approximately solving initial-boundary value problems for multidimensional partial differential equations, the transition to a chain of simpler problems allows constructing economical difference schemes --- splitting with respect to spatial variables. In some cases, we can separate subproblems of different nature --- splitting by physical processes. When targeting parallel architecture computers, regionally additive schemes (domain decomposition schemes) are used. Additive schemes for vector problems, as well as for systems of non-stationary coupled equations, can be distinguished as a separate class.

When solving the Cauchy problem for integro-differential equations, we can separate two operators.
The main operator of the problem is associated with modeling processes without taking into account memory effects. An additional operator is included in the integral term of the evolution equation and describes memory effects.
When transforming the nonlocal problem into a local one by approximating the memory kernel by a sum of exponentials and using implicit difference schemes for the coupled system of equations at the new time level, we have a problem that includes both the main and additional operators. We would like to construct unconditionally stable additive-operator difference schemes that simplify the problem at the new time level.
The paper proposes and investigates splitting schemes for the decomposition of the main operator of the problem. It is noted that the possibilities for achieving the goal of simplifying computations when decomposing the additional operator of the problem are not very large.

The structure of the paper is as follows.
Section 2 formalizes the Cauchy problem without memory and with memory, introduces additive representations of operators and describes basic splitting schemes. Section 3 transforms the nonlocal problem into a local one by approximating the kernel by a sum of exponentials and proves the stability of the extended system. Section 4 is devoted to the study of stability of two-level schemes for the vector problem and the construction of splitting schemes. Section 5 extends the approach to the equation with memory of the derivative of the solution. Section 6 summarizes the results of the work.

\section{Problem formulation}\label{sec:2}

The goal is to construct additive-operator schemes (decomposition schemes, splitting schemes) for first-order evolution equations taking into account memory effects.
The key feature of such models is that the equation includes integral (nonlocal) terms.
In our consideration, we focus on standard splitting schemes for the approximate solution of the Cauchy problem for evolution equations.

\subsection{Problem without memory effects} 

We consider the Cauchy problem for a first-order evolution equation:
\begin{equation}\label{2.1}
 \frac{d u}{d t} + A u = f(t),
 \quad 0 < t \leq T, 
\end{equation} 
\begin{equation}\label{2.2}
 u(0)= u^0 .
\end{equation}
We seek a solution $u(t)$ of equation \eqref{2.1} for $0 < t \leq T$ from a finite-dimensional Hilbert space $H$ with the given initial condition \eqref{2.2}.
Under fairly general conditions, the constant (independent of $t$) operator $A$ in (\ref{2.1}) is assumed to be non-self-adjoint and nonnegative.
For its additive decomposition we have
\begin{equation}\label{2.3}
 A = \sum_{\alpha=1}^{p} A_\alpha ,
 \quad A_\alpha \ge 0,
 \quad \alpha = 1,2, \ldots, p .
\end{equation} 

In $H$, the inner product for $u, v \in H$ is $(u,v)$, and the norm is $\|u\| = (u,u)^{1/2}$.
For a self-adjoint and positive operator $D$, we define the Hilbert space
$H_D$ with the inner product and norm $(u,v)_D = (D u,v), \ \|u\|_D = (u,v)_D^{1/2}$.  
For the solution of problem (\ref{2.1})--(\ref{2.3}) we have a simple a priori estimate in $H_D$ for $D = I$:
\begin{equation}\label{2.4}
 \|u(t)\| \leq \|u^0\| + \int_{0}^{t} \|f(s)\| \, d s,
 \quad 0 < t \leq T .
\end{equation} 
Such estimates of stability of the solution with respect to initial data and right-hand side are inherited (see, e.g., \cite{Samarskii1989,SamarskiiMatusVabischevich2002}) when using two-level implicit time approximations.

We introduce a uniform, for simplicity, time grid with step $\tau$ and let $y^n = y(t^n), \ t^n = n \tau$,
$n = 0,1, \ldots, N, \ N\tau = T$.
For the numerical solution of problem (\ref{2.1}), (\ref{2.2}) we use a two-level scheme with weights ($\sigma = \mathrm{const}$):
\begin{equation}\label{2.5}
 \frac{y^{n+1} - y^{n}}{\tau } + A (\sigma y^{n+1} + (1-\sigma) y^{n}) = f^{n+\sigma} , 
 \quad n = 0,1, \ldots, N-1 , 
\end{equation} 
\begin{equation}\label{2.6}
 y^0 = u^0 .
\end{equation} 
using the notation
\[
t^{n+\sigma} = \sigma t^{n+1} + (1-\sigma) t^n,
\quad f^{n+\sigma} = \sigma f^{n+1} + (1-\sigma) f^n .
\]
The difference scheme (\ref{2.5}), (\ref{2.6}) approximates problem (\ref{2.1}), (\ref{2.2}) with second order in $\tau$ for $\sigma = 1/2$ (symmetric scheme) and with first order for $\sigma \neq 1/2$.
The scheme with weight (\ref{2.5}), (\ref{2.6}) is unconditionally stable at $\sigma \geq 1/2$ in $H$.  
For the solution, the estimate 
\begin{equation}\label{2.7}
 \|y^{n+1}\| \leq \| u^0 \| + \sum_{k=0}^{n} \tau \| f^{k+\sigma}\|,
 \quad \ n = 0,1, \ldots, N-1 ,
\end{equation} 
holds.

To reduce the computational work associated with the complexity of the operator $A$ when finding the approximate solution at the new time level, a general methodological approach to studying systems based on analysis and synthesis is used.
At the decomposition stage (analysis), we identify simpler particular subproblems, the study of which at the composition stage (synthesis) allows us to obtain solutions to the overall problem.
This computational technology of analysis (decomposition) and synthesis (composition) for solving non-stationary problems is implemented using splitting schemes \cite{vabishchevich2024computational}.

For the Cauchy problem (\ref{2.1}), (\ref{2.2}) we use the additive representation of the operator $A$ as the sum of simpler operators (\ref{2.3}), so that 
\begin{equation}\label{2.8}
 \frac{d u}{d t} + \sum_{\alpha =1}^{p}A_{\alpha} y = f(t)
\end{equation} 
with some chosen additive representation of the right-hand side:
\[
 f(t) = \sum_{\alpha =1}^{p}f_{\alpha}(t) .
\] 
For the numerical solution of problem (\ref{2.2}), (\ref{2.2}), (\ref{2.8}) various splitting schemes are used \cite{VabishchevichAdditive}.
We note the main classes of splitting schemes and formulate the conditions for their unconditional stability for a general multi-component decomposition ($p > 1$ in (\ref{2.2})).

For a general multi-component decomposition ($p > 1$ in (\ref{2.8})), we can consider component-wise splitting schemes \cite{Marchuk1990,Yanenko1967}.
The approximate solution at the new time level is found by successively solving auxiliary problems:
\begin{equation}\label{2.9}
\begin{split}
 \frac{y^{n+\alpha/p} - y^{n+(\alpha-1)/p}} {\tau}
  & + A_\alpha (\sigma y^{n+\alpha/p}
   + (1-\sigma)y^{n+(\alpha-1)/p}) \\
  &  = f^{n+\sigma}_\alpha,
  \quad \alpha = 1,2,\ldots,p, \quad n = 0,1,\ldots, N-1.
\end{split}
\end{equation}
Unconditional stability of scheme (\ref{2.3}), (\ref{2.6}), (\ref{2.9}) is established in $H$ under the usual restrictions on the weight $\sigma \geq 1/2$. 
The component-wise splitting scheme has first order accuracy in $\tau$. 
For $\sigma = 1/2, \ \alpha = 1,2,\ldots,p$ and the computational procedure according to the Fryazinov--Strang scheme \cite{Fryazinov1968,Strang1968} 
\[
 A_1 \rightarrow A_2 \rightarrow \cdots \rightarrow A_p \rightarrow A_p \rightarrow \cdots \rightarrow A_1 ,
\] 
second order accuracy is established for smooth solutions.  

We also note regularized additive schemes \cite{VabishchevichAdditive,SamarskiiVabischevich1998},
where no intermediate problems or auxiliary functions are introduced and the original equation is approximated directly. 
In this case, the approximate solution is determined from the equation
\begin{equation}\label{2.10}
  \frac{y^{n+1} - y^{n}} {\tau} + \sum_{\alpha=1}^{p} (I + \sigma \tau A_\alpha )^{-1} A_\alpha y^{n}
  = f^{n+\sigma} .
\end{equation} 
Sufficient conditions for unconditional stability of scheme (\ref{2.3}), (\ref{2.6}), (\ref{2.10}) in $H$ are $\sigma \geq p/2$. 
Computational implementation can be ensured by averaging the solutions of auxiliary problems:
\begin{equation}\label{2.11}
\begin{split}
  \frac{y^{n+1}_\alpha - y^{n}} {p \tau} + 
   & A_\alpha (\sigma y^{n+1}_\alpha
   + (1-\sigma)y^{n}) = f^{n+\sigma}_\alpha,  \\
  & \alpha = 1,2,...,p, \quad n = 0,1,\ldots, N-1 ,\\
  & y^{n+1} = \frac 1p \sum_{\alpha=1}^{p} y^{n+1}_\alpha .
\end{split}
\end{equation}
In this case, scheme (\ref{2.11}) is an additive-averaged scheme \cite{VabishchevichAdditive,GordezianiMeladze1974}. 
Among other splitting schemes for problem (\ref{2.2}), (\ref{2.2}), (\ref{2.8}), we can also note vector additive schemes \cite{Abrashin1990,VabishchevichVector}.

\subsection{Problem with memory effects} 

A typical nonlocal problem is associated with the integro-differential equation
\begin{equation}\label{2.12}
 \frac{d u}{d t} + A u + \int_0^t k(t-s) B u(s) d s = f(t),
 \quad 0 < t \leq T .
\end{equation} 
Memory effects are taken into account by specifying the kernel $k(t)$.

Evolutionary equations with memory (\ref{2.12}) are usually considered under the assumption that the kernel $k(t)$ is positive definite \cite{GripenbergBook1990}.
In this case, for all $T > 0$ the kernel $k(t)$ belongs to $L_1(0, T)$ and the inequality
\begin{equation}\label{2.13}
 \int_{0}^{T} \psi(t) \int_{0}^{t} k(t-s) \psi(s) d s \, d t \geq 0,
 \quad \psi \in C[0,T] ,
\end{equation} 
holds. A sufficient condition for the positive definiteness of the kernel $k(t)$ is (see, e.g., \cite{halanay1965asymptotic}):
\begin{equation}\label{2.14}
 k(t) \geq 0,
 \quad \frac{d k}{d t}  (t) \leq 0,
 \quad \frac{d^2 k}{d t^2}  (t) \geq 0,
 \quad t > 0 . 
\end{equation} 

We assume that the operator $B$ is constant, self-adjoint and nonnegative. Similarly to (\ref{2.3}), we have the additive representation:
\begin{equation}\label{2.15}
 B = \sum_{\beta=1}^{q} B_\beta ,
 \quad B_\beta = B_\beta^* \ge 0,
 \quad \beta  = 1,2, \ldots, q .
\end{equation} 

The complexity of computational algorithms for the approximate solution of the Cauchy problem (\ref{2.2}), (\ref{2.12}) is primarily due to the nonlocality of the equation. When determining the solution at a new time level, we must work with the solution at all previous time instants. We focus on using transformations of the original nonlocal problem to a local one by means of additional quantities when using approximations of the memory kernel $k(t)$ by a sum of exponentials. 

In time, unconditionally stable implicit difference approximations are used. The computational implementation is associated with solving equations that include the operators $A$ and $B$.
Simplification of the problem at the new level is achieved by using splitting schemes that isolate simpler problems taking into account decompositions (\ref{2.3}) and (\ref{2.15}).
Possible splitting options can be conveniently illustrated at the differential level by noting the corresponding time approximations when using component-wise splitting schemes of type (\ref{2.9}). When solving the Cauchy problem for the equation
\[
 \frac{d u}{d t} + \sum_{\gamma=1}^{d} D_\gamma u = \sum_{\gamma=1}^{d} f_\gamma (t),
 \quad 0 < t \leq T ,
\]
multi-component splitting schemes correspond to a chain of simpler problems
\[
  \frac {d v_\gamma}{dt} +
  D_\gamma v_\gamma = f_\gamma (t),
  \quad t^{n} < t \le t^{n+1},
  \quad \gamma =1,2,\ldots,d,
\]
\[
  v_1(0) = u^0,
  \quad  v_1(t^n) =  v_{d}(t^n),
  \quad  v_\gamma(t^{n}) = v_{\gamma-1}(t^{n+1}), 
  \gamma = 2,3,\ldots,d.
\]

For problem (\ref{2.2}) and (\ref{2.12}) we can perform a two-component decomposition, separating local and nonlocal subproblems: the operator $A$ and the operator $B$, respectively, when
\[
 D_1 v = A v,
 \quad  D_2 v = \int_0^t k(t-s) B v(s) d s .
\]
The second main possibility is to simplify the problem based on splitting (\ref{2.3}) and isolating the nonlocal part. 
In this case, we have $d = p+1$, and
\[
 D_\gamma = A_\gamma, 
 \quad \gamma = 1,2, \ldots, d-1,
 \quad D_d v = \int_0^t k(t-s) B v(s) d s .
\]
The general case corresponds to splitting both operator $A$ according to (\ref{2.3}) and operator $B$ according to (\ref{2.15}):
\[
 D_\gamma = A_\gamma, 
 \quad \gamma = 1,2, \ldots, p,
\]
\[
 D_\gamma v = \int_0^t k(t-s) B_\gamma v(s) d s ,
 \quad \gamma = p + 1, p+2, \ldots, p+q = d .
\]

Starting from the above problems on individual time intervals, it is necessary to construct component-wise splitting schemes.
The key feature of the nonlocal problems under consideration is the choice of convenient approximations of the integral operator terms for computational implementation.

\section{Local problem for the extended system of equations}\label{sec:3}

The transition from the nonlocal problem to a local one is ensured by approximating the memory kernel by a sum of exponentials. 
In this case, we have a system of equations that is extended by adding ordinary differential equations for auxiliary functions. Unconditional stability of operator-difference schemes is established using standard two-level time approximations.

\subsection{Transformation of the problem} 

For the approximate solution of problems with memory, we focus (see, e.g., \cite{vabMemory,vabishchevich2023approximate}) on the transition from a nonlocal problem to a local one.
Such a transformation is primarily associated with approximating the kernel by a sum of exponentials.
In order not to complicate the text with non-essential technical details, we restrict ourselves to considering equation (\ref{2.12}) with $f(t) = 0$.

The problem of approximating functions by a sum of exponentials is among the most studied problems of nonlinear approximation \cite{braess2012nonlinear}, but of course requires separate attention. It is not considered in this paper and we assume that the memory kernel $k(t)$ in equation (\ref{2.12}) is represented as
\begin{equation}\label{3.1}
 k(t) = \sum_{i=1}^{m} a_i \exp(-b_i t),
 \quad t \geq 0 . 
\end{equation} 
Moreover, we work under the assumption that 
\begin{equation}\label{3.2}
 a_i > 0,
 \quad b_i > 0,
 \quad   i = 1,2,\ldots, m .
\end{equation} 
Under constraints (\ref{3.2}), the sufficient conditions for positive definiteness (see (\ref{2.14})) of the function $k(t)$ are satisfied. 

Similarly to \cite{vabMemory}, we introduce auxiliary functions
\begin{equation}\label{3.3}
 u_i(t) = \int_{0}^{t} \exp(-b_i (t-s)) u(s) d s ,
 \quad i = 1,2, \ldots, m .  
\end{equation} 
This allows us to write equation (\ref{2.12}) for $f(t) = 0$ as
\begin{equation}\label{3.4}
 \frac{d u}{d t} + A u + \sum_{i=1}^{m} a_i B u_i = 0 . 
\end{equation} 
For $u_i(t), \ i = 1,2, \ldots, m$ from (\ref{3.3}) we have the equations
\begin{equation}\label{3.5}
 \frac{d u_i}{d t} + b_i u_i - u = 0 ,
 \quad i = 1,2, \ldots, m . 
\end{equation} 
The extended system of equations (\ref{3.4}), (\ref{3.5}) is supplemented by initial conditions
\begin{equation}\label{3.6}
  u(0) = u^0, 
 \quad u_i(0) = 0,
 \quad i = 1,2, \ldots, m .  
\end{equation} 

\begin{theorem}\label{t-1}
Let in (\ref{2.12}) the operators $A \ge 0, \ B = B^* > 0$. 
Then for the solution of problem (\ref{3.2}), (\ref{3.4})--(\ref{3.6}) the following stability estimate with respect to initial data holds:
\begin{equation}\label{3.7}
 \|u(t)\|^2 + \sum_{i=1}^{m} a_i \|u_i(t)\|^2_B \leq  \|u^0\|^2,
 \quad 0 < t \leq T .
\end{equation} 
\end{theorem}
	
\begin{proof}
Multiply equation (\ref{3.4}) scalarly in $H$ by $u(t)$, and the individual equations (\ref{3.5}) by $a_i B u_i$ and add them. Taking into account the assumptions about the operators $A, \ B$, this gives the inequality
\[
 \frac{d}{dt} \Big (\|u(t)\|^2 + \sum_{i=1}^{m} a_i \|u_i(t)\|^2_B \Big ) \le 0,
\]
from which the desired a priori estimate (\ref{3.7}) follows.
\end{proof}

\subsection{Vector formulation} 

The study of time approximations can be conveniently performed by writing equations (\ref{3.5}) in the form
\begin{equation}\label{3.8}
 a_i B \frac{d u_i}{d t} + a_i b_i B u_i - a_i B u = 0 ,
 \quad i = 1,2, \ldots, m . 
\end{equation} 
Let us write the system (\ref{3.4}), (\ref{3.8}) as a single vector evolution equation of first order.
Define the vector of unknowns
$\bm u = \{u, u_1, \ldots, u_m \}$ and from (\ref{3.4}), (\ref{3.6}), (\ref{3.8}) go to the Cauchy problem
\begin{equation}\label{3.9}
  \bm B \frac{d \bm u}{d t} + \bm A \bm u = 0  ,
\end{equation} 
\begin{equation}\label{3.10}
 \bm u(0) = \bm u^0 ,
\end{equation} 
where $\bm u^0 = \{u^0, 0, \ldots, 0 \}$.
For the operator matrices $\bm B$ and $\bm A$ we have the representation
\begin{equation}\label{3.11}
\begin{split}
 \bm B & = \mathrm{diag} \, \big (I,a_1 B, \ldots, a_m B \big),
\\
\bm A & = \left (\begin{array}{cccc}
  A  & a_1 B  & \cdots &  a_m B   \\
  - a_1 B   &  a_1 b_1 B  & \cdots &  0 \\
  \cdots  & \cdots & \cdots &  0 \\
  - a_m B &  0 & \cdots &  a_m b_m B \\
\end{array}
 \right ) .
\end{split}
\end{equation} 

Problem (\ref{3.9})--(\ref{3.11}) is considered on the direct sum of spaces $\bm H = H \oplus \ldots  \oplus H$,
where for $\bm u, \bm v \in \bm H$ the inner product and norm are defined by
\[
 (\bm u, \bm v) = (u, v) + \sum_{1=1}^{m} (u_i, v_i),
 \quad \|\bm u\| =  (\bm u, \bm u)^{1/2} .
\]
Taking into account (\ref{2.3}) and (\ref{2.15}), we have
\begin{equation}\label{3.12}
 \quad \bm B = \bm B^* > 0,
 \quad \bm A \ge 0 .
\end{equation} 

To obtain the stability estimate, multiply equation (\ref{3.9}) scalarly in $\bm H$ by $\bm u(t)$.
With (\ref{3.12}) this gives 
\[
  \frac{d}{d t} \| \bm u\|^2_{\bm B} \le 0 .
\]
Taking into account the initial conditions (\ref{3.10}), we obtain
\begin{equation}\label{3.13}
 \|\bm u(t)\|^2_{\bm B} \leq 
 \|\bm u^0\|^2_{\bm B} .
\end{equation} 
In our case
\[
\begin{split}
 \|\bm u(t)\|^2_{\bm B} & = \|u(t)\| + \sum_{i=1}^{m} a_i \|u_i(t)\|_{B}^2  , \\
 \|\bm u^0\|^2_{\bm B} & = \|u^0\|^2 ,
\end{split}
\] 
so the a priori estimate (\ref{3.13}) for the vector problem (\ref{3.9})--(\ref{3.11}) gives estimate (\ref{3.7}) for the solution of problem (\ref{3.4})--(\ref{3.6}).

\section{Time approximations}\label{sec:4}

The stability of two-level weighted schemes for solving the Cauchy problem for the formulated vector equation is investigated.
Standard component-wise splitting schemes are constructed with an additive splitting of the operator $\bm A$ based on representation (\ref{2.3}).
The possibilities of constructing unconditionally stable time approximations when splitting the operator of the integral term of the problem are also discussed.

\subsection{Weighted schemes} 

For the approximate solution of the Cauchy problem (\ref{3.9}), (\ref{3.10}) we use the usual two-level weighted schemes with $\sigma = \operatorname{const}$, where 
\begin{equation}\label{4.1}
 \bm B \frac{\bm y^{n+1} - \bm y^{n}}{\tau } + \bm A \bm y^{n+\sigma} = 0,
 \quad n = 0,1,\ldots, N-1 ,
\end{equation} 
\begin{equation}\label{4.2}
 \bm y^0 = \bm u^0 ,
\end{equation} 
using the notation
\[
 \bm y^{n+\sigma} = \sigma \bm y^{n+1} + (1-\sigma ) \bm y^{n} ,
 \quad   \bm y^{n} = \{y^{n}, y_1^{n}, \ldots, y_m^{n}\} .
\] 

The difference scheme (\ref{4.1}), (\ref{4.2}) approximates problem (\ref{3.9}), (\ref{3.10}) with sufficient smoothness of the solution $\bm u(t)$ with first order in $\tau$ for $\sigma \neq 0.5$ and with second order for $\sigma = 0.5$ (Crank-Nicolson scheme).
Based on the general stability (well-posedness) conditions for operator-difference schemes \cite{Samarskii1989,SamarskiiMatusVabischevich2002}, the following statement is established. The peculiarity of the problem under consideration is related to the nonnegativity and non-self-adjointness of the operator $\bm A$ (see (\ref{3.12})).

\begin{theorem}\label{t-2}
Under conditions (\ref{3.12}) the two-level scheme (\ref{4.1}), (\ref{4.2}) is unconditionally stable for $\sigma \geq 0.5$.
Under these constraints, for an approximate solution to the problem, the a priori estimate 
\begin{equation}\label{4.3}
 \|\bm y^{n+1}\|^2_{\bm B} \leq \|\bm u^0\|^2,
 \quad n = 0,1,\ldots ,  N-1 ,
\end{equation} 
holds.
\end{theorem}
	
\begin{proof}
Multiply equation (\ref{4.1}) scalarly in $\bm H$ by $\tau \bm y^{n+\sigma}$.
With $\bm A \geq 0$ we obtain
\begin{equation}\label{4.4}
 (\bm B (\bm y^{n+1} - \bm y^{n}), \bm y^{n+\sigma}) \leq 0 .
\end{equation} 
The following statement holds (see \cite{vabishchevich2013flux}).
\begin{lemma}\label{l-1}
Let 
\[
  w = \sigma u + (1-\sigma) v  
\] 
and the constant $\sigma \geq 0.5$ for $u,v$ from some Hilbert space $H_D$ ($D=D^* >0$).
Then
\[
  (D(u-v),w) \geq (\|u\|_D - \|v\|_D) \|w\|_D .
\] 
\end{lemma}

Taking into account that $\bm B = \bm B^* > 0$, from (\ref{4.4}) and Lemma \ref{l-1} we obtain 
\[
 \|\bm y^{n+1}\|^2_{\bm B} \leq \|\bm y^{n}\|^2_{\bm B} .
\]
From this inequality the desired estimate (\ref{4.3}) follows.
\end{proof}

When applying the weighted scheme (\ref{4.1}), (\ref{4.2}) we have corresponding problems for individual solution components.  
In this case, for the approximate solution of problem (\ref{3.4})–(\ref{3.6}) we use the difference scheme  
\begin{equation}\label{4.5}
 \frac{y^{n+1} - y^{n}}{\tau } + A y^{n+\sigma}
 + \sum_{i=1}^{m} a_i B y_i^{n+\sigma} = 0,
\end{equation} 
\begin{equation}\label{4.6}
 \frac{y_i^{n+1} - y_i^{n}}{\tau} + b_i y_i^{n+\sigma } - y^{n+\sigma} = 0 ,
 \quad i = 1,2, \ldots, m ,
 \quad n = 0,1, \ldots ,   N-1 ,
\end{equation} 
\begin{equation}\label{4.7}
 y^{0} = u^0,
 \quad y_i^{0} = 0,
 \quad i = 1,2, \ldots, m . 
\end{equation} 
For the approximate solution of problem (\ref{4.5})–(\ref{4.7}), inequality (\ref{4.3}) gives the a priori estimate  
\begin{equation}\label{4.8}
 \|y^{n+1}\|^2 + \sum_{i=1}^{m} a_i \|y_i^{n+1}\|_{B}^2  \leq \|u^0\|^2 ,
 \quad n = 0,1, \ldots .
\end{equation}  
Estimate (\ref{4.8}) is a discrete analogue of estimate (\ref{3.7}) for the solution of the differential problem (\ref{3.4})--(\ref{3.6}).  

To find the solution at the new time level, it is convenient to write equations (\ref{4.5}), (\ref{4.6}) in the form  
\begin{equation}\label{4.9}
 \frac{y^{n+\sigma} - y^{n}}{\sigma\tau } + A y^{n+\sigma}
 + \sum_{i=1}^{m} a_i B y_i^{n+\sigma} = 0,
\end{equation} 
\begin{equation}\label{4.10}
 \frac{y_i^{n+\sigma} - y_i^{n}}{\sigma\tau} + b_i y_i^{n+\sigma } - y^{n+\sigma} = 0 ,
 \quad i = 1,2, \ldots, m ,
 \quad n = 0,1, \ldots ,  N-1.
\end{equation} 
From equation (\ref{4.10}) we obtain
\begin{equation}\label{4.11}
 y_i^{n+\sigma} = \sigma \tau r_i  y^{n+\sigma} + \chi_i^{n},
 \quad r_i = \frac{1}{1 + \sigma b_i \tau } ,
 \quad \chi_i^{n} = r_i y_i^{n} ,
 \quad i = 1,2,\ldots, m .
\end{equation}  
Substituting this into equation (\ref{4.9}) gives the equation for finding $y^{n+\sigma}$:
\begin{equation}\label{4.12}
 \Big (1 + \sigma \tau A + \sigma \tau \mu B \Big ) y^{n+\sigma} = y^{n} + \chi^{n} ,
\end{equation} 
where
\[
  \mu = \sum_{i=1}^{m} a_i r_i , 
  \quad \chi^{n} = y^{n} - \sum_{i=1}^{m} a_i B \chi_i^{n} .
\] 
The transition to the next time level $n+1$ is ensured by solving problem (\ref{4.12}) for $y^{n+\sigma}$ and computing the auxiliary quantities $y_i^{n+\sigma}, \ i = 1,2,\ldots, m,$ according to (\ref{4.11}).  
After that, we set
\[
\begin{split}
  y^{n+1} & = \frac{1}{\sigma} (y^{n+\sigma} - (1-\sigma)y^{n}), \\
  y_i^{n+1} & = \frac{1}{\sigma} (y_i^{n+\sigma} - (1-\sigma)y_i^{n}),   
   \quad i = 1,2,\ldots, m .
\end{split}
\]
In the considered problem with memory, it is necessary to solve $m$ simple auxiliary local evolution problems by explicitly computing their solutions at the new time level.
The main computational work is associated (see (\ref{4.12})) with solving the problem with the operator $I + \sigma \tau A + \sigma \tau \mu B$. A reduction in computational work can be achieved, in particular, by separating subproblems with individual operators $A$ and $B$, and their components in representations (\ref{2.3}) and (\ref{2.15}).

\subsection{Splitting schemes} 

Standard splitting schemes for the approximate solution of problem (\ref{3.9}), (\ref{3.10}) under conditions (\ref{3.12}) are based (see, e.g., \cite{Marchuk1990,VabishchevichAdditive}) on the additive representation of the operator $\bm A$:
\begin{equation}\label{4.13}
 \bm A = \sum_{\gamma=1}^{d} \bm A_{\gamma} ,
 \quad \bm A_{\gamma} \ge 0 ,
 \quad \gamma = 1,2, \ldots, d .
\end{equation} 
When using component-wise splitting schemes, the approximate solution is determined by sequentially solving problems for individual operator terms $\bm A_{\gamma} \ge 0, \  \gamma = 1,2, \ldots, d$ in (\ref{4.13}):
\begin{equation}\label{4.14}
\begin{split}
 \bm B \frac{\bm y^{n+\gamma/d} - \bm y^{n+(\gamma-1)/d}}{\tau} & + 
 \bm A_\gamma \big (\sigma \bm y^{n+\gamma/d} + (1-\sigma) \bm y^{n+(\gamma-1)/d} \big ) = 0, \\
 \quad \gamma & = 1,2, \ldots, d  ,
 \quad n = 0,1, \ldots, N-1 .
\end{split}
\end{equation} 

\begin{theorem}\label{t-3}
The two-level operator-difference scheme (\ref{3.12}), (\ref{4.2}), (\ref{4.14}) with additive representation (\ref{4.13}) of the operator $\bm A$ is unconditionally stable for $\sigma \geq 0.5$ in $\bm H_{\bm B}$. The stability estimate is:
\begin{equation}\label{4.15}
 \|\bm y^{n+1}\|^2_{\bm B} \leq \|\bm u^0\|^2_{\bm B},
 \quad n = 0,1,\ldots,  N-1.
\end{equation} 
\end{theorem}

\begin{proof}
Multiply each equation (\ref{4.14}) scalarly in $\bm H$ by $\tau \sigma \bm y^{n+\gamma/d} + (1-\sigma) \bm y^{n+(\gamma-1)/d}$. Taking into account $\bm A_\gamma \geq 0$ and Lemma~\ref{l-1}, similarly to the proof of Theorem~\ref{t-2} we obtain the inequalities
\[
\|\bm y^{n+\gamma/d}\| \le \|\bm y^{n+(\gamma-1)/d}\| ,
 \quad \gamma = 1,2, \ldots, d  ,
\]
from which estimate (\ref{4.15}) follows.
\end{proof}
 
Various versions of splitting schemes (\ref{3.12}), (\ref{4.2}), (\ref{4.13}), (\ref{4.14}) are associated with specifying the operator terms $\bm A_{\gamma}, \ \gamma = 1,2, \ldots, d$ in the additive representation (\ref{4.13}).
It is natural to separate the memory effects, when (see (\ref{3.11}))
\[
\begin{split}
 \bm A_1 & = \mathrm{diag} \, \big (A, 0, \ldots, 0 \big),
\\
\bm A_2 & = {\left (\begin{array}{cccc}
  0  & a_1 B  & \cdots &  a_m B   \\
  - a_1 B   &  a_1 b_1 B  & \cdots &  0 \\
  \cdots  & \cdots & \cdots &  0 \\
  - a_m B &  0 & \cdots &  a_m b_m B \\
\end{array}
 \right ) . }
\end{split}
\] 
For the problems we consider with $A \ge 0$, $B = B^* > 0$, the conditions $\bm A_1 \ge 0, \ \bm A_2 \ge 0$ are obviously satisfied. 

At the first half-step ($d=2$) we solve the problem without taking memory into account, which is taken into account in the second half-step. Taking into account that $y_i^{n+1/2} = y_i^{n}, \ i = 1,2, \ldots, m$, the computational formulas are written as:
\[
 \frac{y^{n+1/2} - y^{n}}{\tau } + A (\sigma y^{n+1/2} + (1-\sigma)y^{n}) = 0,
\]
\[
 \frac{y^{n+1} - y^{n+1/2}}{\tau } + \sum_{i=1}^{m} a_i B (\sigma y_i^{n+1} + (1-\sigma)y_i^{n}) = 0 ,
\]
\[
 \frac{y_i^{n+1} - y_i^{n}}{\tau} + b_i (\sigma y_i^{n+1} + (1-\sigma)y_i^{n})  - \sigma y^{n+1} - (1-\sigma)y^{n+1/2} = 0 ,
 \quad n = 0,1, \ldots ,   N-1 ,
\]
\[
 y^{0} = u^0,
 \quad y_i^{0} = 0,
 \quad i = 1,2, \ldots, m . 
\]
With this, the transition to the new time level is implemented as follows.
\begin{enumerate}
\item First, given the known right-hand side, $y^{n+1/2}$ is found from the equation
\[
  (1 + \sigma \tau A) y^{n+1/2} = \varphi .
\]
\item Then, for $y^{n+1}$ using the notation introduced earlier (see (\ref{4.12})), the equation 
\[
  (1 + \sigma \tau \mu B) y^{n+1/2} = \psi 
\] 
is solved.
\item From the found $y^{n+1/2}$ and $y^{n+1}$, the $y_i^{n+1}, \ i = 1,2, \ldots, m$ are computed.
\end{enumerate}

A more general case corresponds to splitting the operator $A$.
For decomposition (\ref{2.3}) set ($d = p+1$)
\[
\begin{split}
 \bm A_\gamma & = \mathrm{diag} \, \big (A_\gamma, 0, \ldots, 0 \big),
 \quad \gamma = 1,2, \ldots, p ,
\\
\bm A_{p+1} & = {\left (\begin{array}{cccc}
  0  & a_1 B  & \cdots &  a_m B   \\
  - a_1 B   &  a_1 b_1 B  & \cdots &  0 \\
  \cdots  & \cdots & \cdots &  0 \\
  - a_m B &  0 & \cdots &  a_m b_m B \\
\end{array}
 \right ) . }
\end{split}
\] 
The features of such splitting are associated with the approximations
\[
  (1 + \sigma \tau A_\gamma) y^{n+\gamma/d} = \varphi_\gamma,
   \quad \gamma = 1,2, \ldots, p .
\]
In this case, instead of solving the equation with the operator $A$, $p$ problems with operators $A_\gamma, \ \gamma = 1,2, \ldots, p$ are solved sequentially.

Constructing splitting schemes based on decomposition (\ref{2.3}) of the operator $B$ is difficult. The problem is that in this case we must (see (\ref{3.11})) split not only the operator $\bm A$ in equation (\ref{3.9}), but also the operator $\bm B$. The theory and computational practice of constructing additive operator-difference schemes with splitting of the operator at the time derivative are not very well developed.
Some possibilities are noted in \cite{vabishchevich2012new,vabishchevich2024splitting}. Based on the available results, it is not very clear how to construct decomposition-composition schemes for our problem (\ref{3.9})--(\ref{3.11}).

Consider a special case where the operator $B$ is
\begin{equation}\label{4.16}
  B = C^* C,
  \quad C = \sum_{\beta = 1}^{q} C_\beta .
\end{equation}
with constant operators $C_\beta, \ \beta = 1,2, \ldots, q$.
Previously, splitting schemes with such a factorized operator were constructed \cite{vabishchevich2022splitting} for second-order evolution equations.
Instead of (\ref{3.3}), we define the auxiliary functions as follows:
\[
 u_i(t) = \int_{0}^{t} \exp(-b_i (t-s)) C u(s) d s ,
 \quad i = 1,2, \ldots, m .    
\]
The main equation (compare with (\ref{3.3})) is written as
\begin{equation}\label{4.17}
 \frac{d u}{d t} + A u + \sum_{i=1}^{m} a_i C^* u_i = 0 . 
\end{equation} 
For $u_i(t), \ i = 1,2, \ldots, m$ we obtain
\begin{equation}\label{4.18}
 \frac{d u_i}{d t} + b_i u_i - Cu = 0 ,
 \quad i = 1,2, \ldots, m . 
\end{equation} 
Taking into account the initial conditions (\ref{3.6}) for the solution we obtain the a priori estimate
\[
 \|u(t)\|^2 + \sum_{i=1}^{m} a_i \|u_i(t)\|^2 \leq  \|u^0\|^2,
 \quad 0 < t \leq T .
\]

Problem (\ref{3.6}), (\ref{4.17}), (\ref{4.18}) is written in vector form (\ref{3.9}), (\ref{3.10}) where now
\begin{equation}\label{4.19}
\begin{split}
\bm A & = \left (\begin{array}{cccc}
  A  & a_1 C^*  & \cdots &  a_m C^*   \\
  - a_1 C   &  a_1 b_1 I  & \cdots &  0 \\
  \cdots  & \cdots & \cdots &  0 \\
  - a_m C &  0 & \cdots &  a_m b_m I \\
\end{array}
 \right ) , \\
\bm B & = \mathrm{diag} \, \big (I,a_1 I, \ldots, a_m I \big) .
\end{split}
\end{equation} 
The key point that allows constructing splitting schemes based on (\ref{2.3}), (\ref{4.16}) is that now the operator $\bm B$ in (\ref{4.19}) does not depend on either $A$ or $C$.

Component-wise splitting schemes (\ref{3.12}), (\ref{4.2}), (\ref{4.14}) are constructed taking into account (\ref{4.19}) based on the general decomposition of the operator $\bm A$:
\[
\begin{split}
 \bm A_\gamma & = \mathrm{diag} \, \big (A_\gamma, 0, \ldots, 0 \big),
 \quad \gamma = 1,2, \ldots, p ,
\\
\bm A_{\gamma} & = {\left (\begin{array}{cccc}
  0  & a_1 C_{\gamma-p}^*  & \cdots &  a_m C_{\gamma-p}^*   \\
  - a_1 C_{\gamma-p}   &  a_1 b_1 / q \, I  & \cdots &  0 \\
  \cdots  & \cdots & \cdots &  0 \\
  - a_m C_{\gamma-p} &  0 & \cdots &  a_m b_m / q \, I\\
\end{array}
 \right ) , \quad \gamma = p+1, p+2, \ldots,  d = p+q . }
\end{split}
\] 
The nonnegativity conditions for the operators $\bm A_\gamma, \ \gamma = 1,2, \ldots, p$ are easily verified.

\section{Other problems}\label{sec:5}

Splitting schemes for other evolution problems with memory are constructed similarly.
We will not deviate too much from the main content of the paper and will not consider, for example, second-order evolution equations or systems of equations.
We will only note the possibilities associated with equations in which the integral term describes memory effects not of the solution itself, but of the time derivative of the solution.
 
\subsection{Differential problem} 

Together with (\ref{2.12}) we consider the integro-differential equation 
\begin{equation}\label{5.1}
 \frac{d u}{d t} + A u + \int_0^t k(t-s) B \frac{d u}{d s}(s) d s = 0,
 \quad 0 < t \leq T ,
\end{equation} 
which takes into account the history of the time derivative of the solution.
In \cite{vabishchevich2023approximate}, it is assumed that the operators $A$ and $B$ are self-adjoint and positive:
\begin{equation}\label{5.2}
 A = A^* > 0 ,
 \quad B = B^* > 0 .
\end{equation} 

The transformation of the nonlocal Cauchy problem (\ref{2.2}), (\ref{5.1}), (\ref{5.2}) to a local one is ensured by representing the memory kernel $k(t)$ as a sum of exponentials according to (\ref{3.1}), (\ref{3.2}).
We introduce (see \cite{vabishchevich2023approximate}) auxiliary functions 
\begin{equation}\label{5.3}
 u_i(t) = \int_{0}^{t} \exp(-b_i (t-s)) \frac{d u}{d s}(s) d s ,
 \quad i = 1,2, \ldots, m .  
\end{equation} 
In this case, equation (\ref{5.1}) is again written in the form (\ref{3.4}), and for $u_i(t), \ i = 1,2, \ldots, m$ from (\ref{5.3}) we obtain
\begin{equation}\label{5.4}
 \frac{d u_i}{d t} + b_i u_i - \frac{d u}{d t} = 0 ,
 \quad i = 1,2, \ldots, m . 
\end{equation} 

Taking into account that 
\[
  u_i = \frac{1}{b_i} \frac{d u}{d t} - \frac{1}{b_i} \frac{d u_i}{d t} ,
\]
we write equation (\ref{3.4}) in the form
\begin{equation}\label{5.5}
 \Big (1 + \sum_{i=1}^{m} \frac{a_i}{b_i} B \Big ) \frac{d u}{d t} + A u   -  \sum_{i=1}^{m} \frac{a_i}{b_i} B \frac{d u_i}{d t} = 0.
\end{equation} 
The Cauchy problem (\ref{3.6}), (\ref{5.4}), (\ref{5.5}) is written in vector form (\ref{3.9}), (\ref{3.10}) with 
\begin{equation}\label{5.6}
\begin{split}
 \bm A & = \mathrm{diag} \, \big (A, a_1 B, \ldots, a_m B \big),
\\
\bm B & = \left (\begin{array}{cccc}
  I + {\displaystyle \sum_{i=1}^{m} \frac{a_i}{b_i} B  } & - {\displaystyle \frac{a_1}{b_1} B}  & \cdots &  - {\displaystyle \frac{a_m}{b_m} B}   \vspace{2mm} \\
  - {\displaystyle \frac{a_1}{b_1} B }  &  {\displaystyle \frac{a_1}{b_1} B } & \cdots &  0 \\
  \cdots  & \cdots & \cdots &  0 \\
  - {\displaystyle \frac{a_m}{b_m} B } &  0 & \cdots &  {\displaystyle \frac{a_m}{b_m} B } \\
\end{array}
 \right ) .
\end{split}
\end{equation} 

Under assumptions (\ref{5.2}) we obtain
\begin{equation}\label{5.7}
 \bm B = \bm B^* \geq 0,
 \quad \bm A = \bm A^* > 0 .
\end{equation} 
We multiply scalarly in $\bm H$ equation (\ref{3.9}) by $d \bm u / dt$. Taking into account the properties of the operators (\ref{5.7}), we obtain
\[
 0 = \Big (\bm B \frac{d \bm u}{d t}, \frac{d \bm u}{d t} \Big ) + \frac{1}{2} \frac{d}{d t} \|\bm u\|^2_{\bm A}  
  \le \frac{1}{2} \frac{d}{d t} \|\bm u\|^2_{\bm A} .
\]
Thus we arrive at the estimate
\begin{equation}\label{5.8}
 \|\bm u(t)\|^2_{\bm A} \leq \|\bm u^0\|^2_{\bm A} ,
  \quad 0 < t \leq T .
\end{equation} 
For the individual solution components we have
\[
 \|\bm u(t)\|^2_{\bm A} = \|u(t)\|^2_A + \sum_{i=1}^{m} a_i \|u_i(t)\|_{B}^2 , 
 \quad \|\bm u^0\|^2_{\bm A} = \|u^0\|^2_A .
\] 

In this paper we consider the more general case of nonnegativity of the operator $A$ --- see (\ref{2.3}). Under these conditions, instead of (\ref{5.7}) we have
\begin{equation}\label{5.9}
 \bm B = \bm B^* \geq 0,
 \quad \bm A \ge 0 .
\end{equation} 
We cannot obtain standard a priori estimates in spaces generated by such operators. Instead of the stability estimate (\ref{5.9}) for all solution components, we restrict ourselves to the estimate
\begin{equation}\label{5.10}
 \|u(t)\|^2 \leq \|u^0\|^2 ,
  \quad 0 < t \leq T .
\end{equation} 

To prove (\ref{5.10}), multiply equation (\ref{3.9}) by $\bm u(t)$. Taking into account conditions (\ref{5.9}) we obtain
\[
  \frac{d} {d t} (\bm B \bm u,\bm u ) \le 0 .
\]
Consequently, we have
\[
 (\bm B \bm u(t),\bm u(t)) \le (\bm B \bm u^0,\bm u^0 ) .
\]
Taking into account (see (\ref{5.6}))
\[
  (\bm B \bm u(t),\bm u(t)) \ge \|u(t)\|^2
\]
and the initial conditions (\ref{3.6}), this gives (\ref{5.10}).
It is precisely the stability estimate (\ref{5.10}) for the solution of the problem without a priori estimates of the auxiliary solutions that we will focus on when choosing time approximations.
 
\subsection{Two-level operator-difference schemes} 

For simplicity, we restrict ourselves to the case of symmetric approximations, where the weight parameter $\sigma = 1/2$. The more general case $\sigma \ge 1/2$ is considered similarly. The approximate solution of problem (\ref{3.9}), (\ref{3.10}) under conditions (\ref{5.9}) is determined from the equation
\begin{equation}\label{5.11}
 \bm B \frac{\bm y^{n+1} - \bm y^{n}}{\tau } + \bm A \frac{\bm y^{n+1} + \bm y^{n}}{2}   = 0,
 \quad n = 0,1,\ldots, N-1 ,
\end{equation} 
with initial conditions (\ref{4.2}).

\begin{theorem}\label{t-4}
Let in scheme (\ref{4.2}), (\ref{5.6}), (\ref{5.11}) the operators $A \ge 0$, $B = B^* > 0$. Then the scheme is unconditionally stable and the a priori estimate 
\begin{equation}\label{5.12}
 \|y^{n+1}\|^2 \leq \|u^0\|^2,
 \quad n = 0,1,\ldots ,  N-1 ,
\end{equation} 
holds.
\end{theorem}
	
\begin{proof}
Multiplying equation (\ref{5.11}) by $\tau (\bm y^{n+1} + \bm y^{n})/2$, we obtain
\[
  (\bm B \bm y^{n+1}, \bm y^{n+1}) \le (\bm B \bm y^{0}, \bm y^{0}) .
\]
Similarly to (\ref{5.10}), from this inequality we obtain the desired a priori estimate (\ref{5.12}).
\end{proof}

Additive operator-difference schemes are constructed based on splitting the operator $A$ in the form (\ref{2.3}). In this case,
\begin{equation}\label{5.13}
 \bm A_\alpha = \mathrm{diag} \, \big (A_\alpha, a_1 B, \ldots, a_m B \big),
 \quad \alpha = 1,2, \ldots, p .
\end{equation} 
When using the component-wise splitting scheme, the approximate solution is determined from
\begin{equation}\label{5.14}
\begin{split}
 \bm B \frac{\bm y^{n+\alpha/p} - \bm y^{n+(\alpha-1)/p}} {\tau}
  & + \bm A_\alpha \frac{\bm y^{n+\alpha/p} + \bm y^{n+(\alpha-1)/p}}{2} = 0, \\
  \alpha & = 1,2,\ldots,p, \quad n = 0,1,\ldots, N-1.
\end{split}
\end{equation}
Similarly to Theorem \ref{t-4}, the following statement is established.

\begin{theorem}\label{t-5}
For $A \ge 0$, $B = B^* > 0$ the scheme (\ref{4.2}), (\ref{5.6}), (\ref{5.14}) is unconditionally stable and the a priori estimate (\ref{5.12}) holds.
\end{theorem}

When splitting the operator $B$ in accordance with (\ref{2.15}), the construction of standard splitting schemes is not obvious. 
This is because with such a decomposition we must focus on additive operator schemes for equation (\ref{3.9}) with simultaneous splitting of the operators $\bm A$ and $\bm B$. 

\section{Conclusions}\label{sec:6} 

\begin{enumerate}
\item The Cauchy problem for a first-order integro-differential equation is considered. The kernel is assumed to be a difference kernel, and the equation itself includes two operators in a finite-dimensional Hilbert space. The problems of numerical solution of such problems are mainly related to the need to operate with the solution for all previous time instants. By representing the memory kernel as a sum of exponentials, the nonlocal in time problem is transformed into a local problem for a system of weakly coupled evolution equations with additional ordinary differential equations for auxiliary functions. 
\item The initial value problem is formulated in vector form on the direct sum of Hilbert spaces. Unconditional stability of two-level operator-difference schemes with weights is proved. Splitting schemes are proposed and investigated by separating the local and integral operators of the problem. The possibilities of deeper decomposition of the approximate solution under an additive representation of the operators of the integro-differential equation are noted. 
\item The possibility of constructing similar splitting schemes for other nonlocal problems is noted. The Cauchy problem for a first-order integro-differential equation with memory of the time derivative of the solution is considered. 
Problems with memory effects for second-order evolution equations and systems of equations can be considered similarly.
\end{enumerate}

\begin{funding}
The work is supported by North-Caucasus Center for Mathematical Research under agreement no.~075-02-2022-892 with the Ministry of Science and Higher Education of the Russian Federation and was carried out with financial support from the Russian Science Foundation (project  no.~24-11-00058).
\end{funding}

\providecommand{\bysame}{\leavevmode\hbox to3em{\hrulefill}\thinspace}
\providecommand{\MR}{\relax\ifhmode\unskip\space\fi MR }
\providecommand{\MRhref}[2]{%
  \href{http://www.ams.org/mathscinet-getitem?mr=#1}{#2}
}
\providecommand{\href}[2]{#2}


\begin{thebibliography}{10}

\bibitem{Abrashin1990}
V.~N. Abrashin, {A variant of the method of variable directions for the
  solution of multidimensional problems of mathematical-physics. I.},
  \emph{Differ. Equations} \textbf{26} (1990), 314--323, in Russian.

\bibitem{braess2012nonlinear}
D.~Braess, \emph{Nonlinear Approximation Theory}, Springer Science \& Business
  Media, 2012.

\bibitem{ChenBook1998}
C.~Chen and T.~Shih, \emph{Finite Element Methods for Integrodifferential
  Equations}, World Scientific, Singapore, 1998.

\bibitem{LionsBook}
R.~Dautray and J.-L. Lions, \emph{Mathematical Analysis and Numerical Methods
  for Science and Technology}, ~1, Springer, 2000.

\bibitem{evans2010partial}
L.~C. Evans, \emph{Partial Differential Equations}, American Mathematical
  Society, 2010.

\bibitem{Fryazinov1968}
I.~V. Fryazinov, Economical symmetrized schemes for solving boundary value
  problems for multi-dimensional parabolic equation, \emph{Zh. Vychisl. Mat.
  Mat. Fiz.} \textbf{8} (1968), 436--443, in Russian.

\bibitem{GordezianiMeladze1974}
D.~G. Gordeziani and G.~V. Meladze, On modeling of third boundary value problem
  for a multi-dimensional parabolic equations in an arbitrary domain by
  one-dimensional equations, \emph{Zh. Vychisl. Mat. Mat. Fiz.} \textbf{14}
  (1974), 246--250, in Russian.

\bibitem{GripenbergBook1990}
G.~Gripenberg, S.-O. Londen and O.~Staffans, \emph{Volterra Integral and
  Functional Equations}, Cambridge University Press, 1990.

\bibitem{halanay1965asymptotic}
A~Halanay, On the asymptotic behavior of the solutions of an
  integro-differential equation, \emph{Journal of Mathematical Analysis and
  Applications} \textbf{10} (1965), 319--324.

\bibitem{KnabnerAngermann2003}
P.~Knabner and L.~Angermann, \emph{Numerical Methods for Elliptic and Parabolic
  Partial Differential Equations}, Springer Verlag, 2003.

\bibitem{linz1985analytical}
P.~Linz, \emph{Analytical and Numerical Methods for Volterra Equations}, SIAM,
  1985.

\bibitem{Marchuk1990}
G.~I. Marchuk, \emph{Splitting and alternating direction methods}, Handbook of
  Numerical Analysis, Vol. I (P.~G. Ciarlet and J.-L. Lions, eds.),
  North-Holland, 1990, pp.~197--462.

\bibitem{mclean1993numerical}
W.~McLean and V.~Thom{\'e}e, Numerical solution of an evolution equation with a
  positive-type memory term, \emph{The ANZIAM Journal} \textbf{35} (1993),
  23--70.

\bibitem{mclean1996discretization}
W.~McLean, V.~Thom{\'e}e and L.~B. Wahlbin, Discretization with variable time
  steps of an evolution equation with a positive-type memory term,
  \emph{Journal of Computational and Applied Mathematics} \textbf{69} (1996),
  49--69.

\bibitem{pruss2013evolutionary}
J.~Pr{\"u}ss, \emph{Evolutionary Integral Equations and Applications},
  Birkh{\"a}user, 2013.

\bibitem{QuarteroniValli}
A.~Quarteroni and A.~Valli, \emph{Numerical Approximation of Partial
  Differential Equations}, Springer, 2008.

\bibitem{Samarskii1989}
A.~A. Samarskii, \emph{The Theory of Difference Schemes}, Marcel Dekker, New
  York, 2001.

\bibitem{SamarskiiMatusVabischevich2002}
A.~A. Samarskii, P.~P. Matus and P.~N. Vabishchevich, \emph{Difference Schemes
  with Operator Factors}, Kluwer Academic Pub, 2002.

\bibitem{SamarskiiVabischevich1998}
A.~A. Samarskii and P.~N. Vabishchevich, Regularized additive full
  approximation schemes, \emph{Dokl. Akad. Nauk} \textbf{358} (1998), 461--464,
  in Russian.

\bibitem{Strang1968}
G.~Strang, On the construction and comparison of difference schemes, \emph{SIAM
  Journal on Numerical Analysis} \textbf{5} (1968), 506--517.

\bibitem{vabishchevich2012new}
P.~Vabishchevich, On a new class of additive (splitting) operator-difference
  schemes, \emph{Mathematics of Computation} \textbf{81} (2012), 267--276.

\bibitem{VabishchevichVector}
P.~N. Vabishchevich, Vector additive difference schemes for first-order
  evolution equations, \emph{Comput. Math. Math. Phys.} \textbf{36} (1996),
  317--322.

\bibitem{VabishchevichAdditive}
P.~N. Vabishchevich, \emph{Additive Operator-Difference Schemes: Splitting
  Schemes}, Walter de Gruyter GmbH, Berlin, Boston, 2013.

\bibitem{vabishchevich2013flux}
P.~N. Vabishchevich, Flux-splitting schemes for parabolic equations with mixed
  derivatives, \emph{Computational Mathematics and Mathematical Physics}
  \textbf{53} (2013), 1139--1152.

\bibitem{vabMemory}
P.~N. Vabishchevich, Numerical solution of the {C}auchy problem for {V}olterra
  integrodifferential equations with difference kernels, \emph{Applied
  Numerical Mathematics} \textbf{174} (2022), 177--190.

\bibitem{vabishchevich2022splitting}
P.~N. Vabishchevich, Splitting Schemes for Some Second-Order Evolutionary
  Equations, \emph{International Journal of Numerical Analysis and Modeling}
  \textbf{19} (2022), 19--32.

\bibitem{vabishchevich2023approximate}
P.~N. Vabishchevich, Approximate solution of the Cauchy problem for a
  first-order integrodifferential equation with solution derivative memory,
  \emph{Journal of Computational and Applied Mathematics} \textbf{422} (2023),
  114887.

\bibitem{vabishchevich2024computational}
P.~N. Vabishchevich, Computational decomposition and composition technique for
  approximate solution of nonstationary problems, \emph{Journal of
  Computational and Applied Mathematics} \textbf{451} (2024), 116111.

\bibitem{vabishchevich2024splitting}
P.~N. Vabishchevich, Splitting Schemes with Additive Representation of the
  Operator at the Time Derivative, \emph{Moscow University Computational
  Mathematics and Cybernetics} \textbf{48} (2024), 1--6.

\bibitem{vabishchevich2025operator}
P.~N. Vabishchevich, Operator-Difference Schemes for Systems of First-Order
  Integro-Differential Equations, \emph{Differential equations} \textbf{61}
  (2025), 1051--1059.

\bibitem{Yanenko1967}
N.~N. Yanenko, \emph{The Method of Fractional Steps. The Solution of Problems
  of Mathematical Physics in Several Variables}, Springer, 1971.

\end{thebibliography}
\end{document}